\documentclass[11pt]{amsart}
\usepackage[utf8]{inputenc}
\usepackage{amsmath,amsthm}
\usepackage{epsfig,psfrag,graphicx,sidecap}
\usepackage[pdftex,pdfborder={0 0 0},colorlinks]{hyperref}
\usepackage{latexsym,lmodern,amssymb}
\usepackage{dsfont,txfonts}
\usepackage{amsfonts,enumerate}
\usepackage{variations}
\usepackage{fontenc}
\usepackage{amsaddr}

\newtheorem{theorem}{Theorem}[section]

\newtheorem{corollary}[theorem]{Corollary}

\theoremstyle{definition}

\numberwithin{equation}{section}

\newcommand{\ent}{\operatorname{Ent}}

\begin{document}

\title[Stability of the log-Sobolev inequality for the Gaussian measure]{Remark on the stability of the log-Sobolev inequality for the Gaussian measure}

\author[F. Feo]{Filomena Feo}
\address{Dipartimento di Ingegneria, Universit\`a degli Studi di Napoli ``Parthenope"\\
Centro Direzionale Isola C4\\
 80100 Naples, Italy.}
\email{filomena.feo@uniparthenope.it}

\author[M.R. Posteraro]{Maria Rosaria Posteraro}
\address{Dipartimento di Matematica e Applicazioni ``Renato Caccioppoli"\\
 Via Cintia - Complesso Monte  S. Angelo\\
  80100 Naples, Italy.}
\email{posterar@unina.it}

\author[C. Roberto]{Cyril Roberto}
\address{Universit\'e Paris Ouest Nanterre la D\'efense\\
 MODAL'X EA 3454, 200 avenue de la R\'epublique\\
 92000 Nanterre, France.}
\email{croberto@math.cnrs.fr}

\date{\today}

\thanks{Supported by the grants ANR 2011 BS01 007 01,  ANR 10 LABX-58, ANR11-LBX-0023-01  and by  the Gruppo Nazionale per l'Analisi Matematica, la Probabilit\' a e le loro Applicazioni (GNAMPA) of the Istituto Nazionale di Alta Matematica (INdAM)}



\begin{abstract}
In this  note we bound the deficit in the logarithmic Sobolev Inequality  and in the Talagrand transport-entropy Inequality for the Gaussian measure, in any dimension, by mean of a distance introduced by Bucur and Fragal\`a.
\end{abstract}

\maketitle


\section{introduction}
The log-Sobolev inequality asserts that, in any dimension $n$ and for any smooth enough function
$f \colon  \mathbb{R}^n \to \mathbb{R}_+^*:=(0,+\infty)$, it holds
\begin{equation} \label{losob}
\ent_{\gamma_n} (f) \leq \frac{1}{2} \int_{\mathbb{R}^n} \frac{|\nabla f|^2}{f} \, d\gamma_n,
\end{equation}
where $\gamma_n(dx)=\varphi_n(x)dx:=(2 \pi)^{-\frac{n}{2}} \exp\{-\frac{|x|^2}{2}\}dx$, $x \in \mathbb{R}^n$, is the standard Gaussian measure with density $\varphi_n$, $|x|=\sqrt{\sum_{i=1}^n x_i^2}$ stands for the Euclidean norm of $x=(x_1,\dots,x_n)$ (accordingly $|\nabla f|$ is the Euclidean length of the gradient) and $\ent_{\gamma_n}(f):=\int_{\mathbb{R}^n} f \log f \,  d\gamma_n - \int_{\mathbb{R}^n} f \, d\gamma_n \log \int_{\mathbb{R}^n} f \, d\gamma_n$ is the entropy of $f$ with respect to $\gamma_n$. The constant $1/2$ is optimal. Moreover, equality holds in \eqref{losob} if and only if $f$ is the exponential of a linear function, \textit{i.e.}\ there exist $a \in \mathbb{R}^n$, $b \in \mathbb{R}$ such that $f(x)=\exp\{a \cdot x + b\}$, $x \in \mathbb{R}^n$. For simplicity we may write $\varphi$ and $\gamma$ for $\varphi_1$ and $\gamma_1$.

The log-Sobolev inequality above goes back to Stam \cite{stam} in the late fifties. Later Gross, in his seminal paper \cite{gross-75}, rediscovered the inequality and proved its fundamental equivalence with the so-called hypercontractivity property, a notion used by Nelson \cite{nelson1} in quantum filed theory.
Since then the log-Sobolev inequality attracted a lot of attention with many developments, applications and connections with other fields, including Geometry, Analysis, Combinatorics, Probability Theory and Statistical Mechanics. We refer to the monographs \cite{ane,bakry,gross-93,villani,ledoux-99,martinelli} for an introduction.
Finally, we mention that equality cases, in \eqref{losob}, appear in the paper by Carlen \cite{carlen}.

Very recently there has been some interest in the study of the stability of the log-Sobolev inequality \eqref{losob}. Namely the question is: can one bound the difference between the right and left hand side of \eqref{losob} in term of the distance (in a sense to be defined) between $f$ and the set of optimal functions?
In other words, can one bound from below the \emph{deficit}
\begin{equation} \label{deficit}
\delta_{LS}(f):= \frac{1}{2} \int_{\mathbb{R}^n} \frac{|\nabla f|^2}{f} \, d\gamma_n - \ent_{\gamma_n}(f)
\end{equation}
in some reasonable way?
We refer to \cite{indrei,bobkov,fathi} for various results in this direction.
What is however currently lacking in the aforementioned literature is a result stating that, in fact,
$\delta_{LS}(f) \geq d(f, \mathcal{O})$ where $\mathcal{O}:=\{e^{a \cdot x + b}, a \in \mathbb{R}^n, b \in \mathbb{R} \}$ is the set of functions achieving equality in \eqref{losob} and where $d$ is some distance.

Our aim is to give a result in this direction, using a distance introduced by Bucur and Fragal\`a in \cite{bucur} that we recall now.


\subsubsection*{Bucur and Fragal\`a's construction of a distance modulo translation}

In this   section we recall the procedure of Bucur and Fragal\`a \cite{bucur} to define a distance (modulo translation) in dimension $n$ starting with a distance (modulo translation) in dimension 1. We first give the definition of a \emph{distance modulo translation}.

Let $\mathcal{S}_n$ be some set of non-negative functions defined on $\mathbb{R}^n$.
A mapping $m : \mathcal{S}_n \times \mathcal{S}_n \to \mathbb{R}_+$ is said to be a distance modulo translation (on $\mathcal{S}_n$) if $(i)$ $m$ is symmetric, $(ii)$ it satisfies the triangular inequality and $(iii)$ $m(u,v)=0$ iff there exists $a \in \mathbb{R}^n$ such that $v(x)=u(x+a)$ for all $x \in \mathbb{R}^n$.

Now, given a direction $\xi \in \mathbb{S}^{n-1}$ (the unit sphere of $\mathbb{R}^n$), let $x=(x',t\xi)$ be the decomposition of any point $x \in \mathbb{R}^n$ in the direct sum of the linear span of $\xi$ and its orthogonal hyperplane $H_\xi:=\{y \in \mathbb{R}^n : <\xi , y> =0\}$ (here $< \cdot,\cdot>$ stands for the Euclidean scalar product).
Then, for all integrable function $f \colon \mathbb{R}^n \to \mathbb{R}$, define
$f_\xi \colon \mathbb{R} \to \mathbb{R}, t \mapsto f_\xi(t):=\int_{H_\xi} f(x',t\xi)d\mathcal{H}^{n-1}(x')$,
where $\mathcal{H}^{n-1}$ is the $(n-1)$-dimensional Hausdorff measure on $H_\xi$.
Given a distance $m$ modulo translation on some set $\mathcal{S}$ of non-negative real functions, set
$$
\mathcal{S}_n:=\{
f\in L^{1}(\mathbb{R}^n,\mathbb{R}_+), \> xf(x)\in L^{1}(\mathbb{R}^n,\mathbb{R}^n), \> f_\xi \in \mathcal{S} \mbox{ for all } \xi \in \mathbb{S}^{n-1}\}
$$
\color{black}
and, for $f, g \in \mathcal{S}_n$,
$$
m_n(f,g):= \sup_{\xi \in \mathbb{S}^{n-1}} m(f_\xi ,g_\xi) .
$$
In \cite[Corollary 2.3]{bucur}, it is proved that $m_n$ is a distance modulo translation on $\mathcal{S}_n$.

Also, Bucur and Fragal\`a \cite{bucur} introduce the following distance modulo translation that we may use in the next sections. Set
\begin{align*}
\mathcal{B}:= &\left\{ u \colon \mathbb{R} \to \mathbb{R}_+^* : \mbox{continuous and }
\int_{\mathbb{R}} u(x) dx=1 \right\} .
\end{align*}
Given two probability measures $\mu(dx)=u(x)dx$, and $\nu(dx)=v(x)dx$, $u,v \in \mathcal{B}$, set
$T=F_\nu^{-1} \circ F_\mu$, where $F_\mu(x):=\int_{-\infty}^x u(y)dy$
and $F_\nu(x):=\int_{-\infty}^x v(y)dy$ are the distribution functions of $\mu$ and $\nu$ respectively (observe that, since $u \in \mathcal{B}$, $F_\mu^{-1}$ is well defined and so does $T'$ (note that $T$ is increasing)).
$T$ is the transport map that pushes forward $\mu$ onto $\nu$, \textit{i.e.} the mapping satisfying
$\int_{\mathbb{R}} h(T)\, d\mu=\int_{\mathbb{R}} h\, d\nu$ for all bounded continuous function $h$. The following is a distance modulo translation on $\mathcal{B}$ (see \cite[Proposition 3.5]{bucur})
\begin{equation}\label {d}
d(u,v):= \int \frac{|1-T'|}{\max(1,T')} d\mu .
\end{equation}
We denote by $d_n$ and $\mathcal{B}_n$ the distance modulo translation and the set of functions constructed by the above procedure, starting from $d$ and $\mathcal{B}$ in dimension 1.


\section{Stability of the Log-Sobolev inequality}

 In order to state our main result, we need first to give a precise statement for
\eqref{losob} to hold.

It is well-known that \eqref{losob} holds for any $f$ such that  $\int_{\mathbb{R}^n} | f |  d\gamma_n +   \int_{\mathbb{R}^n}  \frac{|\nabla f|^2}  {f}  d\gamma_n < \infty ,  $  {\it i.e.}  $|f|^{1/2} \in H^1(\gamma_n)$, see  \textit{e.g.}\ \cite[Chapter 1]{bogachev}. By a density argument one can restrict \eqref{losob}, without loss, to all $f$ positive
(since $|\nabla |f||=|\nabla f|$ almost everywhere), and by homogeneity, we can assume furthermore that $\int_{\mathbb{R}^n} f \,  d\gamma_n = 1$.
We call $\mathcal{A}_n$ the set of $\mathcal{C}^1$ functions $f \colon \mathbb{R}^n \to \mathbb{R}_+^*$ such that $\int_{\mathbb{R}^n} f \,  d\gamma_n=1$ and  $\int_{\mathbb{R}^n} \frac{|\nabla f|^2}  {f}  d\gamma_n < \infty$.
It is dense in the set of all functions satisfying the log-Sobolev Inequality \eqref{losob} and is contained in $\mathcal{B}_n$. We observe that the set of extremal functions (with the proper normalization) $\{\exp\{a \cdot x - \frac{|a|^2}{2}\}, a \in \mathbb{R}^n \}$ is contained in $\mathcal{A}_n$.

We are now in position to state our main theorem (recall the definition of $d_n$ from the previous section).

\begin{theorem} \label{thmain}
For all $n$ and all  $f \in \mathcal{A}_n$ it holds
$$
\delta_{LS}(f) \geq \frac{1}{2} d_n \left( f \varphi_n, \varphi_n \right)^2 .
$$
\end{theorem}

Before moving to the proof of Theorem \ref{thmain} which is very short and elementary, let us comment on the above result.

First, from the above result, we (partially) recover the cases of equality in the log-Sobolev inequality for the Gaussian measure \cite{carlen}. Indeed, $f \in \mathcal{A}_n$ achieves the equality in the log-Sobolev inequality iff
$\delta_{LS}(f)=0$ iff $f \varphi_n$ is a translation of $\varphi_n$ iff $f(x)=\exp\{-a \cdot x - \frac{|a|^2}{2}\}$ for some $a \in \mathbb{R}^n$. This is only partial since Theorem \ref{thmain} do not deal with all functions  satisfying the log-Sobolev inequality but only with $f \in \mathcal{A}_n$.
There is in fact some technical issues here: the distance $d$ is no more a distance modulo translation if $F_\mu^{-1}$ ($\mu$ is the one dimensional probability measure with density $f$) is not absolutely continuous (a property that is guaranteed by the fact that, in the definition of $\mathcal{A}_n$, we impose the positivity of the functions), see \cite[Remark 3.6 (i)]{bucur}. Hence, a result involving the distance $d_n$ cannot recover, by essence, the full generality of Carlen's equality cases \cite{carlen}. However, $\mathcal{A}_n$ is very close to cover the set of all functions satisfying the log-Sobolev inequality (in particular it is dense in such a space) and, to the best of our knowledge, there is no result in the current literature that gives a lower bound of the deficit involving a distance without any second moment condition.

The assumption $f$ of class $\mathcal{C}^1$, in the definition of $\mathcal{A}_n$ can certainly be relaxed.
Indeed, one only needs, in dimension 1, that $F_\mu^{-1}$ is an absolutely continuous function \cite[Proposition 3.5]{bucur} (for $d\mu(x)=f(x) \varphi(x) dx$, $x \in \mathbb{R}$). In dimension $n$, such a property should hold for all directions $\xi \in \mathbb{S}^{n-1}$. For this reason, and as mentioned above, there is no hope to obtain the whole family of functions satisfying the log-Sobolev Inequality. Hence, we opted for an easy and clean presentation rather than for a more technical one (a weaker assumption on $f$ would have led us to technical approximations in many places, that, to our opinion, play no essential role).

We also observe that our result does not capture the product character of the log-Sobolev inequality. Indeed, if one considers,  on $\mathbb{R}^n$, a function of the form $f(x)=h(x_1)h(x_2)\dots h(x_n)$, $x=(x_1,\dots,x_n)$, with $h \colon \mathbb{R} \to \mathbb{R}_+^*$, then it is not difficult to see that  $\delta_{LS}(f)$ is of order $n$
(\textit{i.e.} $\delta_{LS}(f)=n\delta_{LS}(h)$)
thanks to the tensorisation property of  \eqref{losob} (see \textit{e.g.}\ \cite[Chapter 1]{ane}), while
$d_n \left(f \varphi_n, \varphi_n \right)$ is of order 1
(\textit{i.e.} $d_n \left(f \varphi_n, \varphi_n \right)=d \left(h \varphi, \varphi \right)$).
 This mainly comes from our use of Bucur and Fragal\`a's quantitative Prekopa-Leindler Inequality which is also, by construction, 1 dimensional. See below for some results based on the tensorisation property of the log-Sobolev inequality.

\begin{proof}[Proof of Theorem \ref{thmain}]
The proof is based on the approach of Bobkov and Ledoux \cite{bobkov-ledoux00} to the log-Sobolev inequality by mean of the Pr\'ekopa-Leindler Inequality, together with an improved version of the Pr\'ekopa-Leindler Inequality of Bucur and Fragal\`a \cite{bucur}. Our starting point is the following result (see \cite[Proposition 3.5]{bucur}): given a triple
$u, v , w \colon \mathbb{R}^n \to \mathbb{R}_+$ with $u,v \in \mathcal{B}_n$ and $\lambda \in [0,1]$ that satisfy $w(\lambda x + (1-\lambda)y) \geq u(x)^\lambda v(y)^{1-\lambda}$ for all $x,y \in \mathbb{R}^n$, it holds
\begin{equation} \label{PL}
\int_{\mathbb{R}^n} w(x) dx - 1 \geq {1\over 2}\,\lambda^{1+\lambda}(1-\lambda)^{2-\lambda}  d_n(u,v)^2 .
\end{equation}
We stress that the constant $\lambda $ in the right hand side of the latter is not given explicitly in \cite{bucur},
but the reader can easily recover such a bound following carefully the proof of \cite[Proposition 3.5]{bucur}.
(Inequality \eqref{PL} goes back to the seventies \cite{leindler,prekopa} and has numerous applications in convex geometry and functional analysis. We refer to the monographs \cite{barthe,gardner,villani} for an introduction.
We further mention that equality cases are given in \cite{dubuc}, and refer to Ball and B\"or\"oczky \cite{ball1,ball2}
for related results on the stability of the Pr\'ekopa-Leindler Inequality.)

Our aim is to apply \eqref{PL} to a proper choice of triple $u,v,w$. Following \cite{bobkov-ledoux00}, let
$f=e^g$ with $g$ sufficiently smooth
with compact support and $\int_{\mathbb{R}^n} f \,  d\gamma_n = 1$,
$\lambda \in (0,1)$, and set
$$
u_\lambda(x)=\frac{e^{\frac{g(x)}{1-\lambda}}\varphi_n(x)}{\int_{\mathbb{R}^n} e^{\frac{g}{1-\lambda}} \,  d\gamma_n},
\quad v(y)=\varphi_n(y)
\quad \mbox{and }
w_\lambda(z)=e^{g_\lambda(z)} \varphi_n(z)
$$
with
$$
g_\lambda(z) : = \sup_{\genfrac{}{}{0pt}{}{x,y :}{(1-\lambda)x+\lambda y=z}} \left( g(x) - \frac{\lambda(1-\lambda)}{2} |x-y|^2\right)
- (1-\lambda) \log \int_{\mathbb{R}^n} e^{\frac{g}{1-\lambda}} \,  d\gamma_n.
$$
The function $g_\lambda$ is the optimal function such that it holds
$w_\lambda((1-\lambda)x+\lambda y) \geq u_\lambda(x)^{1-\lambda}v(y)^\lambda$. Set $h_\lambda(z):=\sup_{\genfrac{}{}{0pt}{}{x,y :}{(1-\lambda)x+\lambda y=z}} \left( g(x) - \frac{\lambda(1-\lambda)}{2} |x-y|^2\right) $.
Then, by \eqref{PL} above, we get
$$
\int_{\mathbb{R}^n} e^{h_\lambda} \, d\gamma_n \geq \left(\int_{\mathbb{R}^n} e^{\frac{g}{1-\lambda}}\right)^{1-\lambda}
\left(1+{1\over 2}\,\lambda^{1+\lambda}(1-\lambda)^{2-\lambda}  d_n(u_\lambda,v)^2\right)
$$
The aim is to take the limit $\lambda \to 0$. We observe that (see \cite{bobkov-ledoux00} for details), as $\lambda$ tends to zero
$$
\left(\int_{\mathbb{R}^n} e^{\frac{g}{1-\lambda}}\right)^{1-\lambda} = \int_{\mathbb{R}^n} e^g  d\gamma_n + \lambda \ent_{\gamma_n}(e^g) + o(\lambda)
$$
and
$$
\int_{\mathbb{R}^n} e^{h_\lambda} \, d\gamma_n = \int_{\mathbb{R}^n} e^g  d\gamma_n + \frac{\lambda}{2(1-\lambda)} \int_{\mathbb{R}^n} |\nabla g|^2 e^g \, d\gamma_n + o(\lambda) .
$$
Therefore, dividing by $\lambda$, and taking the limit, we end up with
$$
\liminf_{\lambda \to 0} {1\over 2}\lambda^\lambda(1-\lambda)^{2-\lambda} d_n(u_\lambda,v)^2 + \ent_{\gamma_n}(e^g) \leq
\frac{1}{2} \int_{\mathbb{R}^n} |\nabla g|^2 e^g   d\gamma_n .
$$
We are left with the study of $\liminf_{\lambda \to 0} d_n(u_\lambda,v)^2$,
since $\lim_{\lambda \to 0} \lambda^\lambda(1-\lambda)^{2-\lambda} = 1$.
For simplicity set $u:=u_0$ (\textit{i.e.}\ $u$ is the function $u_\lambda$ defined above with $\lambda=0$).
By the Lebesgue Theorem  we observe that, for any direction $\xi \in \mathbb{S}^{n-1}$,
$\lim_{\lambda \to 0} (u_\lambda)_\xi = u_\xi$. Hence, using again the Lebesgue Theorem
(observe that, in the definition of $d$, $|1-T'|/\max(1,T') \leq 1$)
\begin{align*}
\liminf_{\lambda \to 0} d_n(u_\lambda,v)
& \geq
\sup_{\xi \in \mathbb{S}^{n-1}} \liminf_{\lambda \to 0}
d((u_\lambda)_\xi,v_\xi)
=
\sup_{\xi \in \mathbb{S}^{n-1}}  d_n(u_\xi,(\varphi_n)_\xi)
=
d_n(e^g \varphi_n,\varphi_n) .
\end{align*}
The expected result follows for $g$ sufficiently smooth. The result for a general $f \in \mathcal{A}_n$ follows by an easy approximation argument (using again the monotone convergence Theorem and the Lebesgue Theorem), details are left to the reader.
\end{proof}

Next we derive from Theorem \ref{thmain} a lower bound on the log-Sobolev inequality, in dimension $n$, that involves
 $n$ times the one dimensional distance $d$. Such a result will capture on one hand the product structure of the inequality, but on the other hand the  deficit will no more be bounded by a distance (modulo translation).

We need some notation. Given $x=(x_1,\dots,x_n) \in \mathbb{R}^n$, $i \in \{1,\dots,n\}$ and $y_i \in \mathbb{R}$, set $\bar{x}^i:=(x_1,\dots,x_{i-1},x_{i+1},\dots,x_n)$ and $\bar{x}^i y_i:=(x_1,\dots,x_{i-1},y_i,x_{i+1},\dots,x_n)$ (so that $\bar{x}^i x_i = x$). Then,
for all functions $f \colon \mathbb{R}^n \to \mathbb{R}$ and all $x \in \mathbb{R}^n$, we denote by $f_{\bar{x}^i} \colon \mathbb{R} \to \mathbb{R}$ the one dimensional function defined by $f_{\bar{x}^i}(y_i):=f(\bar{x}^iy_i)$,
$y_i \in \mathbb{R}$ (obviously $f_{\bar{x}^i}(x_i)=f(x)$).

We may prove the following result.

\begin{corollary}
For all $n$ and all  $f \in \mathcal{A}_n$ it holds
$$
\delta_{LS}(f) \geq \frac{1}{2} \sum_{i=1}^n \int_{\mathbb{R}^{n-1}} d \left( f_{\bar{x}^i} \varphi, \varphi \right)^2 d\gamma_{n-1}(\bar{x}^i) .
$$
\end{corollary}

Now, by construction, if $f(x)=h(x_1)h(x_2)\dots h(x_n)$, $x=(x_1,\dots,x_n)$, with $h \colon \mathbb{R} \to \mathbb{R}_+^*$, then both  $\delta_{LS}(f)$ and the right hand side of the latter are  of (the correct) order $n$.

\begin{proof}
The proof uses the tensorisation property of the entropy. It is well known (see \textit{e.g.}\ \cite[Chapter 1]{ane}) that for any $f \colon \mathbb{R}^n \to \mathbb{R}$, it holds
$$
\ent_{\gamma_n}(f) \leq \sum_{i=1}^n \int_{\mathbb{R}^{n-1}} \ent_\gamma(f_{\bar{x}^i}) d\gamma_{n-1}(\bar{x}^i).
$$
Hence, applying Theorem \ref{thmain} $n$ times, we get (since $f_{\bar{x}^i}'(x_i)=\frac{\partial f}{\partial x_i}(x)$)
\begin{align*}
2 \ent_{\gamma_n}(f) & \leq
\sum_{i=1}^n \int_{\mathbb{R}^{n-1}} \int_\mathbb{R} \frac{{f_{\bar{x}^i}'}^2(x_i)}{f_{\bar{x}^i}(x_i)}  d\gamma(x_i) d\gamma_{n-1}(\bar{x}^i)
-
\sum_{i=1}^n \int_{\mathbb{R}^{n-1}} d(f_{\bar{x}^i} \varphi, \varphi)^2 d\gamma_{n-1}(\bar{x}^i) \\
& =
 \int_{\mathbb{R}^n} \frac{|\nabla f|^2}{f} d\gamma_n
-
\sum_{i=1}^n \int_{\mathbb{R}^{n-1}} d(f_{\bar{x}^i} \varphi, \varphi)^2 d\gamma_{n-1}(\bar{x}^i).
\end{align*}
The expected result follows.
\end{proof}

\section{Stability of the Talagrand transport-entropy Inequality}

In this section we bound the deficit in the so-called Talagrand inequality, using again the distance $d_n$ introduced by Bucur and Fragal\`a.
Recall that (see \textit{e.g.}\ \cite{villani1}) the Kantorovich-Wasserstein distance $W_2$ is defined as
$$
W_2(\nu,\mu):= \inf_{\pi} \left( \iint |x-y|^2 \pi(dx,dy) \right)^\frac{1}{2},
$$
where the infimum runs over all couplings $\pi$ on $\mathbb{R}^n \times \mathbb{R}^n$ with first marginal $\nu$ and second marginal $\mu$ (\textit{i.e.}\ $\pi(\mathbb{R}^n , dy)=\mu(dy)$ and $\pi(dx,\mathbb{R}^n)=\nu(dx)$).
Talagrand, in his seminal paper \cite{talagrand-96}, proved the following inequality: for all
probability measure $\nu$  on $\mathbb{R}^n$, absolutely continuous with respect to $\gamma_n$, it holds
\begin{equation} \label{T}
W_2^2(\nu,\gamma_n) \leq 2 H(\nu|\gamma_n),
\end{equation}
where  $H(\nu|\gamma_n):=\int_{\mathbb{R}} \log \frac{d\nu} {d\gamma_n} \,d\nu$ if $\nu <\!< \gamma_n$ (whose density is denoted by $d\nu/d\gamma_n$) and
$H(\nu|\gamma)=+\infty$ otherwise, is the relative entropy of $\nu$ with respect to $\gamma_n$.
Such an inequality, that is usually called Talagrand transport-entropy inequality, is related to Gaussian concentration in infinite dimension \cite{talagrand-96,marton,gozlan} (see the monographs \cite{gozlan-leonard,ledoux-concentration} for an introduction). It is known, since the celebrated work by Otto and Villani \cite{otto-villani}, that the log-Sobolev inequality \eqref{losob} implies the Talagrand inequality \eqref{T} (in any dimension, see \cite{bobkov-gentil-ledoux,W04,BEHM09,LV07,GRS11,GRS13,GRS14,gigli-ledoux} for alternative proofs and extensions).

The stability of \eqref{T} is also studied in \cite{indrei,fathi,bobkov,cordero15}. We may obtain that, as a direct consequence of the transport of mass approach of \eqref{T} by Cordero-Erausquin \cite{cordero} and of the tensorisation property, one can bound from below, as for the log-Sobolev inequality, the deficit in \eqref{T} by the distance $d_n$ defined by Bucur and Fragal\`a.

\begin{theorem} \label{PT}
For all probability measure $\nu$ on $\mathbb{R}^n$ with continuous and positive density
$f\in \mathcal{B}_n $
with respect to the Gaussian measure $\gamma_n$, it holds
\begin{equation} \label{stabT}
\delta_{Tal}(f):=2 H(\nu|\gamma_n) - W_2^2(\nu,\gamma_n) \geq \frac{1}{2} d_n(f \varphi_n,\varphi_n)^2 .
\end{equation}
\end{theorem}

The above result together with
theorem \eqref{thmain} somehow justify the use of the distance $d_n$. As for Theorem \ref{thmain} the bound on the deficit is one dimensional and thus not of the correct order. This fact may become clear to the reader through the proof: we use some tensorisation property but apply a bound on the deficit only to one single coordinate.

\begin{proof}
The proof goes in two steps : we first prove the lower bound of the deficit in dimension 1, then we use a tensorisation procedure.

From \cite{cordero} we can extract the following one dimensional inequality (here $f \colon \mathbb{R} \to \mathbb{R}_+^*$)
$$
\delta_{Tal}(f) \geq \int_{\mathbb{R}} [T' - 1 - \log T']\,  d\gamma ,
$$
where $T=F_\nu^{-1} \circ F_\gamma$ is the push forward of $\gamma$ onto $\nu$. Using that $s-1-\log s \geq \frac{1}{2} \left(\frac{1-s}{\max(1,s)} \right)^2$ and the Cauchy-Schwartz Inequality, we can conclude that
\begin{align*}
\delta_{Tal} (f)
\geq
\frac{1}{2} \int_{\mathbb{R}} \left(\frac{1-T'}{\max(1,T')} \right)^2 \,  d\gamma
\geq
\frac{1}{2}  \left( \int_{\mathbb{R}} \frac{|1-T'|}{\max(1,T')} \,  d\gamma  \right)^2
 =
\frac{1}{2} d( \varphi,f \varphi)^2,
\end{align*}
which ends the proof of the first step (since $d$ is symmetric).

Next, recall the tensorisation property of the Kantorovich-Wasserstein metric and of the relative entropy: for
$f \colon \mathbb{R}^n \to \mathbb{R}$,  $\nu(dx)=f(x)dx$, we have
$$
W_2^2(\nu,\gamma_n) \leq W_2^2(\nu_1,\gamma)
+ \sum_{i=1}^{n-1} \int_{\mathbb{R}^i} W_2^2(\nu_{x_1,\dots,x_i},\gamma) d\gamma_i(x_1,\dots,x_i)
$$
and
$$
H(\nu|\gamma_n) = H(\nu_1|\gamma) + \sum_{i=1}^{n-1} \int_{\mathbb{R}^i} H(\nu_{x_1,\dots,x_i}|\gamma) d\gamma_i(x_1,\dots,x_i),
$$
where we used the desintegration formula
$$
\nu(dx_1,\dots,dx_n) = \nu_1(dx_1) \nu_{x_1}(dx_2)\nu_{x_1,x_2}(dx_3) \times \cdots \times \nu_{x_1,\dots,x_{n-1}}(dx_n) .
$$
Now the supremum defining $d_n(f \varphi_n,\varphi_n)$ is reached at some $\xi \in \mathbb{S}^{n-1}$ that we may assume for simplicity and without loss of generality (since $\gamma_n$ is invariant by rotation) to be the first unit vector of the canonical basis $(1,0,\dots,0)$. Using the tensorisation formulas above, applying the result we obtained in dimension 1, and \eqref{T} $n-1$ times, we thus get
\begin{align*}
W_2^2(\nu,\gamma_n) & \leq
W_2^2(\nu_1,\gamma)
+ \sum_{i=1}^{n-1} \int_{\mathbb{R}^i} W_2^2(\nu_{x_1,\dots,x_i},\gamma) d\gamma_i(x_1,\dots,x_i) \\
& \leq
2 H(\nu_1|\gamma) - \frac{1}{2} d( f_1 \varphi, \varphi)^2 + 2 \sum_{i=1}^{n-1} \int_{\mathbb{R}^i} H(\nu_{x_1,\dots,x_i}|\gamma) d\gamma_i(x_1,\dots,x_i) \\
& =
2 H(\nu |\gamma) - \frac{1}{2} d( f_1 \varphi, \varphi)^2
 =
2 H(\nu |\gamma) - \frac{1}{2} d_n( f \varphi_n, \varphi_n)^2,
\end{align*}
where we set $f_1$ for the density of $\nu_1$ with respect to $\gamma$. By construction $\nu_1$ is the first marginal of $\nu$ so that $f_1 \varphi = (f \varphi_n)_\xi$.
This ends the proof.
\end{proof}

\section*{Acknowledgement}
The authors warmly thank Emmanuel Indrei and Michel Ledoux for useful discussions on the topic of the present paper. The hospitality of MODAL'X, Universit\'e Paris Ouest Nanterre la D\'efense, and of the IMA, Minneapolis, where part of this paper was done, are gratefully acknowledged.


\bibliographystyle{plain}
\bibliography{FPR-logsob}

\def\cprime{$'$}
\begin{thebibliography}{10}

\bibitem{ane}
C.~An{\'e}, S.~Blach{\`e}re, D.~Chafa{\"\i}, P.~Foug{\`e}res, I.~Gentil,
  F.~Malrieu, C.~Roberto, and G.~Scheffer.
\newblock {\em Sur les in{\'e}galit{\'e}s de {S}obolev logarithmiques},
  volume~10 of {\em Panoramas et Synth{\`e}ses}.
\newblock Soci{\'e}t{\'e} {M}ath{\'e}matique de {F}rance, Paris, 2000.

\bibitem{bakry}
D.~Bakry.
\newblock L'hypercontractivit\'e et son utilisation en th\'eorie des
  semigroupes.
\newblock In {\em Lectures on Probability theory. \'Ecole d'{\'e}t{\'e} de
  {P}robabilit{\'e}s de St-Flour 1992}, volume 1581 of {\em Lecture Notes in
  Math.}, pages 1--114. Springer, Berlin, 1994.

\bibitem{ball1}
K.~Ball and K.~B{\"o}r{\"o}czky.
\newblock Stability of the {P}r\'ekopa-{L}eindler inequality.
\newblock {\em Mathematika}, 56(2):339--356, 2010.

\bibitem{ball2}
K.~Ball and K.~B{\"o}r{\"o}czky.
\newblock Stability of some versions of the {P}r\'ekopa-{L}eindler inequality.
\newblock {\em Monatsh. Math.}, 163(1):1--14, 2011.

\bibitem{BEHM09}
Z.~Balogh, A.~Engoulatov, L.~Hunziker, and O.~E. Maasalo.
\newblock Functional inequalities and {H}amilton-{J}acobi equations in geodesic
  spaces.
\newblock Preprint. Available on the ArXiv http://arxiv.org/abs/0906.0476,
  2009.

\bibitem{barthe}
F.~Barthe.
\newblock Autour de l'in\'egalit\'e de {B}runn-{M}inkowski.
\newblock {\em Ann. Fac. Sci. Toulouse Math. (6)}, 12(2):127--178, 2003.

\bibitem{bobkov}
S.~Bobkov.
\newblock Extremal properties of half-spaces for log-concave distributions.
\newblock {\em Ann. Probab.}, 24(1):35--48, 1996.

\bibitem{bobkov-gentil-ledoux}
S.~G. Bobkov, I.~Gentil, and M.~Ledoux.
\newblock Hypercontractivity of {H}amilton-{J}acobi equations.
\newblock {\em J. Math. Pures Appl. (9)}, 80(7):669--696, 2001.

\bibitem{bobkov-ledoux00}
S.~G. Bobkov and M.~Ledoux.
\newblock From {B}runn-{M}inkowski to {B}rascamp-{L}ieb and to logarithmic
  {S}obolev inequalities.
\newblock {\em Geom. Funct. Anal.}, 10(5):1028--1052, 2000.

\bibitem{bogachev}
V.~I. Bogachev.
\newblock {\em Gaussian measures}, volume~62 of {\em Mathematical Surveys and
  Monographs}.
\newblock American Mathematical Society, Providence, RI, 1998.

\bibitem{bucur}
D.~Bucur and I.~Fragal{\`a}.
\newblock Lower bounds for the {P}r\'ekopa-{L}eindler deficit by some distances
  modulo translations.
\newblock {\em J. Convex Anal.}, 21(1):289--305, 2014.

\bibitem{carlen}
E.~A. Carlen.
\newblock Superadditivity of {F}isher's information and logarithmic {S}obolev
  inequalities.
\newblock {\em J. Funct. Anal.}, 101(1):194--211, 1991.

\bibitem{cordero}
D.~Cordero-Erausquin.
\newblock Some applications of mass transport to {G}aussian-type inequalities.
\newblock {\em Arch. Ration. Mech. Anal.}, 161(3):257--269, 2002.

\bibitem{cordero15}
D.~Cordero-Erausquin.
\newblock Transport inequalities for log-concave measures, quantitative forms
  and applications.
\newblock Preprint, 2015.

\bibitem{dubuc}
S.~Dubuc.
\newblock Crit\`eres de convexit\'e et in\'egalit\'es int\'egrales.
\newblock {\em Ann. Inst. Fourier (Grenoble)}, 27(1):x, 135--165, 1977.

\bibitem{fathi}
M.~Fathi, E~Indrei, and M.~Ledoux.
\newblock Quantitative logarithmic sobolev inequalities and stability
  estimates.
\newblock Preprint available arXiv:1410.6922, 2014.

\bibitem{gardner}
R.~J. Gardner.
\newblock The {B}runn-{M}inkowski inequality.
\newblock {\em Bull. Amer. Math. Soc. (N.S.)}, 39(3):355--405, 2002.

\bibitem{gigli-ledoux}
N.~Gigli and M.~Ledoux.
\newblock From log {S}obolev to {T}alagrand: a quick proof.
\newblock {\em Discrete Contin. Dyn. Syst.}, 33(5):1927--1935, 2013.

\bibitem{gozlan}
N.~Gozlan.
\newblock Characterization of {T}alagrand's like transportation-cost
  inequalities on the real line.
\newblock {\em J. Funct. Anal.}, 250(2):400--425, 2007.

\bibitem{gozlan-leonard}
N.~Gozlan and C.~L{\'e}onard.
\newblock Transport inequalities. {A} survey.
\newblock {\em Markov Process. Related Fields}, 16(4):635--736, 2010.

\bibitem{GRS11}
N.~Gozlan, C.~Roberto, and P.-M. Samson.
\newblock A new characterization of {T}alagrand's transport-entropy
  inequalities and applications.
\newblock {\em Ann. Probab.}, 39(3):857--880, 2011.

\bibitem{GRS13}
N.~Gozlan, C.~Roberto, and P.-M. Samson.
\newblock Characterization of {T}alagrand's transport-entropy inequalities in
  metric spaces.
\newblock {\em Ann. Probab.}, 41(5):3112--3139, 2013.

\bibitem{GRS14}
N.~Gozlan, C.~Roberto, and P.-M. Samson.
\newblock Hamilton {J}acobi equations on metric spaces and transport entropy
  inequalities.
\newblock {\em Rev. Mat. Iberoam.}, 30(1):133--163, 2014.

\bibitem{gross-75}
L.~Gross.
\newblock Logarithmic {S}obolev inequalities.
\newblock {\em Amer. J. Math.}, 97:1061--1083, 1975.

\bibitem{gross-93}
L.~Gross.
\newblock Logarithmic {S}obolev inequalities and contractivity properties of
  semi-groups. in {D}irichlet forms. {D}ell'{A}ntonio and {M}osco eds.
\newblock {\em Lect. Notes Math.}, 1563:54--88, 1993.

\bibitem{indrei}
E.~Indrei and D.~Marcon.
\newblock A quantitative log-{S}obolev inequality for a two parameter family of
  functions.
\newblock {\em Int. Math. Res. Not. IMRN}, (20):5563--5580, 2014.

\bibitem{ledoux-99}
M.~Ledoux.
\newblock Concentration of measure and logarithmic {S}obolev inequalities.
\newblock In {\em S\'eminaire de {P}robabilit\'es, {XXXIII}}, volume 1709 of
  {\em Lecture Notes in Math.}, pages 120--216. Springer, Berlin, 1999.

\bibitem{ledoux-concentration}
M.~Ledoux.
\newblock {\em The concentration of measure phenomenon}, volume~89 of {\em
  Math. Surveys and Monographs}.
\newblock American {M}athematical {S}ociety, Providence, {R}.{I}, 2001.

\bibitem{leindler}
L.~Leindler.
\newblock On a certain converse of {H}\"older's inequality.
\newblock In {\em Linear operators and approximation ({P}roc. {C}onf.,
  {O}berwolfach, 1971)}, pages 182--184. Internat. Ser. Numer. Math., Vol. 20.
  Birkh\"auser, Basel, 1972.

\bibitem{LV07}
J.~Lott and C.~Villani.
\newblock Hamilton-{J}acobi semigroup on length spaces and applications.
\newblock {\em J. Math. Pures Appl. (9)}, 88(3):219--229, 2007.

\bibitem{martinelli}
F.~Martinelli.
\newblock Lectures on {G}lauber dynamics for discrete spin models.
\newblock In {\em Lectures on probability theory and statistics. \'Ecole
  d'été de probabilités de St-Flour 1997}, volume 1717 of {\em Lecture Notes
  in Math.}, pages 93--191. Springer, Berlin, 1999.

\bibitem{marton}
K.~Marton.
\newblock Bounding {$\overline d$}-distance by informational divergence: a
  method to prove measure concentration.
\newblock {\em Ann. Probab.}, 24(2):857--866, 1996.

\bibitem{nelson1}
E.~Nelson.
\newblock A quartic interaction in two dimensions.
\newblock In {\em Mathematical Theory of Elementary Particles (Proc. Conf.,
  Dedham, Mass., 1965)}, pages 69--73. M.I.T. Press, Cambridge, Mass., 1966.

\bibitem{otto-villani}
F.~Otto and C.~Villani.
\newblock Generalization of an inequality by {T}alagrand and links with the
  logarithmic {S}obolev inequality.
\newblock {\em J. Funct. Anal.}, 173(2):361--400, 2000.

\bibitem{prekopa}
A.~Pr{\'e}kopa.
\newblock On logarithmic concave measures and functions.
\newblock {\em Acta Sci. Math. (Szeged)}, 34:335--343, 1973.

\bibitem{stam}
A.~J. Stam.
\newblock Some inequalities satisfied by the quantities of information of
  {F}isher and {S}hannon.
\newblock {\em Information and Control}, 2:101--112, 1959.

\bibitem{talagrand-96}
M.~Talagrand.
\newblock Transportation cost for {G}aussian and other product measures.
\newblock {\em Geom. Funct. Anal.}, 6(3):587--600, 1996.

\bibitem{villani1}
C.~Villani.
\newblock {\em Topics in optimal transportation}, volume~58 of {\em Graduate
  Studies in Mathematics}.
\newblock American Mathematical Society, Providence, RI, 2003.

\bibitem{villani}
C.~Villani.
\newblock {\em Optimal transport}, volume 338 of {\em Grundlehren der
  Mathematischen Wissenschaften [Fundamental Principles of Mathematical
  Sciences]}.
\newblock Springer-Verlag, Berlin, 2009.
\newblock Old and new.

\bibitem{W04}
F.~Y. Wang.
\newblock Probability distance inequalities on {R}iemannian manifolds and path
  spaces.
\newblock {\em J. Funct. Anal.}, 206(1):167--190, 2004.

\end{thebibliography}

\end{document}